\documentclass[preprint,12pt]{elsarticle}
\usepackage{amssymb,amsthm,amsmath}
\usepackage{comment}

\newtheorem{remark}{\textit{Remark}}[section]

\newtheorem{lemma}{\textit{Lemma}}[section]
\newtheorem{theorem}{\textit{Theorem}}[section]

\def\b{\boldsymbol}
\def\note#1{ {\sffamily \textcolor{blue}{#1}} }
\def\Reg{\text{Reg}}
\def\Sm{\textrm{Sm}_\Theta}
\def\D{\textrm{D}}

\begin{document}

\begin{frontmatter}

\title{Tridiagonalization of systems of coupled linear differential equations with variable coefficients by a Lanczos-like method.}

\author[A1]{Pierre-Louis Giscard}
\address[A1]{Universit\'e du Littoral C\^{o}te d'Opale, EA2597-LMPA-Laboratoire de Math\'ematiques Pures et Appliqu\'ees Joseph Liouville, Calais, France. Email: giscard@univ-littoral.fr}

\author[A2]{Stefano Pozza}
\address[A2]{Faculty of Mathematics and Physics, Charles University, Sokolovsk\'a 83, 186 75 Praha 8, Czech Republic. Email: pozza@karlin.mff.cuni.cz. Associated member of ISTI-CNR, Pisa, Italy, and member of INdAM-GNCS group, Italy.}

\begin{abstract}
We show constructively that, under certain regularity assumptions, any system of coupled linear differential equations with variable coefficients can be tridiagonalized by a time-dependent Lanczos-like method. The proof we present formally establishes the convergence of the Lanczos-like algorithm and yields a full characterization of algorithmic breakdowns. From there, the solution of the original differential system is available in closed form. This is a key piece in evaluating the elusive ordered exponential function both formally and numerically.
\end{abstract}


\end{frontmatter}

\section{Introduction}
\subsection{Background}
Systems of coupled linear differential equations with non-constant coefficients naturally arise in a variety of contexts in mathematics \cite{Reid63,kwaSiv72,Corless2003,Blanes15,BenEtAll17} and beyond, from engineering to quantum physics \cite{kuvcera73,Abou2003,hached2018,KirSim19,Autler1955,Shirley1965,Lauder1986,Xie2010}. Yet, determining the solutions of such systems both formally and numerically remains surprisingly difficult, their widespread applicability making these difficulties only more pressing. 

The issue is best presented in the language of linear algebra, and so we consider here an $N\times N$ matrix $\mathsf{A}(t')$ depending on the real-time variable $t'\in I\subseteq \mathbb{R}$ encoding the non-constant coupling coefficients of the linear differential system. In this framework, the unique solution $\mathsf{U}(t',t)$ of the system of coupled linear differential equations with non-constant coefficients 
\begin{equation}\label{FundamentalSystem}
\mathsf{A}(t') \mathsf{U}(t',t)=\frac{d}{dt'}\mathsf{U}(t',t), \quad \mathsf{U}(t,t)=\mathsf{Id},\text{ for all }t\in I,
\end{equation}
with $t\leq t' \in I$ and $\mathsf{Id}$ the identity matrix, is called the \emph{time-ordered exponential} of $\mathsf{A}(t')$. Under the assumption that $\mathsf{A}$ commutes with itself at all times, i.e., $\mathsf{A}(\tau_1)\mathsf{A}(\tau_2)-\mathsf{A}(\tau_2)\mathsf{A}(\tau_1)=\b{0}$ for all $\tau_1,\tau_2 \in I$, then the time-ordered exponential is an ordinary matrix exponential $\mathsf{U}(t',t)=\exp\left(\int_t^{t'} \mathsf{A}(\tau)\, \text{d}\tau\right)$. In general, however, $\mathsf{U}$ has no known explicit form in terms of $\mathsf{A}$ and is usually denoted $\mathcal{T}\exp\left(\int_t^{t'} \mathsf{A}(\tau)\, \text{d}\tau\right)$. Determining this object explicitly not only means solving systems of coupled ODEs with variable coefficients but,  remarkably, a general strategy would also provide formal solutions to systems of coupled linear and non-linear partial differential equations \cite{kosovtsov2002introduction,kosovtsov2004chronological,kosovtsov2009formal}.  

In the context of ODEs, only three methods have been devised to calculate ordered exponentials analytically, only one of which is guaranteed to produce an exact answer in a finite number of steps. These are: the Floquet approach, applicable when $\mathsf{A}(t')$ is periodic and which produces an infinite perturbative expansion of the solution\footnote{That is, a series in terms of powers of a parameter that should be small to guarantee convergence.} usually too complicated to be evaluated beyond its first or second terms \cite{Blanes2009}; the Magnus series expansion \cite{Magnus1954}, which presents the solution as the matrix exponential of an increasingly intricate infinite series of nested commutators plagued by incurable divergence issues \footnote{That is its convergence domain is incurably small (even if not so restrictively); see \cite{Casas07} and also \cite{Feldman1984,Maricq1987,IseAl2000,Sanchez2011}. In spite of this, Magnus series are very much in use nowadays in quantum physics \cite{Blanes2009}, because they guarantee the unitary of the approximated solutions even when the series diverges (!).}; and the path-sum approach, which expresses the solution exactly as a continued fraction of finite depth \cite{Giscard2015,GiscardBonhomme2019} but requires solving an NP-hard problem \cite{GiscardPozza2019,Flum2004}.

Recently P.-L. G. and S. P. proposed a \emph{constructive} method to tridiagonalize systems of linear differential equations with non-constant coefficients \cite{GiscardPozza2019,ProceedingsPaper2019}, from which one can easily evaluate $\b{w}^{H}\mathsf{U}(t',t)\b{v}$ for any two vectors $\b{w}, \b{v}$ with $\b{w}^{H}\b{v}=1$. Here $\b{w}^{H}$ denotes the Hermitian transpose of $\b{w}$. Under the assumptions that the coefficients of the tridiagonalized system are ``well-behaved'' distributions (in a sense to be made precise below) and that the method does not breakdown, this approach -- a Lanczos-like algorithm -- is able to produce the tridiagonalization. The purpose of the present work is to prove that such assumptions hold. That is, we establish that it is indeed possible to tridiagonalize a system of coupled linear differential equations with variable coefficients using a time-dependent Lanczos approach, provided the matrix $\mathsf{A}(t')$ is composed of \emph{smooth} functions of $t'$ and there exists at least one $\rho \in I$ so that the matrix $\mathsf{A}(\rho)$ is tridiagonalizable in the usual sense.
At the heart of the strategy employed is a non-commutative convolution-like product, denoted by $\ast$, defined between certain distributions. We therefore begin by recalling the definition and properties of the product utilized before stating and giving the proof on the tridiagonalization.

\subsection{$\ast$-Product}\label{ProdDef}

Let $t$ and $t'$ be two real variables. We consider the class $\D$ of all distributions which are linear superpositions of Heaviside theta functions and Dirac delta derivatives with smooth coefficients. That is, a distribution $d$ is in $\D$ if and only if it can be written as
$$
d(t',t)=\widetilde{d}(t',t)\Theta(t'-t) + \sum_{i=0}^N \widetilde{d}_i(t',t)\delta^{(i)}(t'-t),  
$$
where $N\in \mathbb{N}$ is finite, $\Theta(\cdot)$ stands for the Heaviside theta function (with the convention $\Theta(0)=1$) and $\delta^{(i)}(\cdot)$ is the $i$th derivative of the Dirac delta distribution $\delta\equiv \delta^{(0)}$. Here and from now on, the tilde on $\widetilde{f}$ indicates that $\widetilde{f}(t',t)$ is an ordinary \emph{smooth} function in both variables. 
Note that we consider we consider distributions as Schwartz did \cite{schwartz1952,schwartz1978}, i.e., $f \in \D$ should be interpreted as a linear functional on a set of test functions.

We can endow the class $\D$ with a non-commutative algebraic structure upon defining a product between its elements. For $f_1, f_2
\in \D$ we define the convolution-like $\ast$ product between $f_1(t',t)$ and $f_2(t',t)$ as
\begin{equation}\label{eq:def:*}
  \big(f_2 * f_1\big)(t',t) := \int_{-\infty}^{\infty} f_2(t',\tau) f_1(\tau, t) \, \text{d}\tau.
\end{equation}
From this definition, we find that the $\ast$-product is associative over $\D$, and that the identity element with respect to the $\ast$-product is the Dirac delta distribution, $1_\ast:=\delta(t'-t)$. 
When $f(t',t) = f(t'-t)$ has bounded supporting set, the $\ast$-product $f \ast g$ (and $g \ast f$) is the convolution for distributions defined by Schwartz; see \cite[\S~11]{schwartz1952} and \cite[Chapter~VI]{schwartz1978}. Given that $\delta^{(i)}(t'-t)$ has bounded supporting set, for every $f \in \D$, the $\ast$-product $\delta^{(i)}(t'-t) \ast f$ and $f \ast \delta^{(i)}(t'-t)$ are well-defined and are both elements of $\D$; see \cite{GiscardPozza2019,GisPozInv19} for further details.
Observe that the $*$-product is not, in general, a convolution but may be so when both $f_1(t',t)$ and $f_2(t',t)$ depend only on the difference $t'-t$.\\[-.7em]

As a case of special interest here, we shall also consider the subclass $\Sm$ of $\D$ comprising those distributions which are piecewise smooth, i.e., of the form
\begin{equation}\label{PSmForm}
f(t',t)=\widetilde{f}(t',t)\Theta(t'-t).
\end{equation}
For $f_1,\,f_2\in\Sm$, the $\ast$-product between $f_1,f_2$ simplifies to
\begin{align*}
  \big(f_2 * f_1\big)(t',t) &= \int_{-\infty}^{\infty} \widetilde{f}_2(t',\tau) \widetilde{f}_1(\tau, t)\Theta(t'-\tau)\Theta(\tau-t) \, \text{d}\tau,\\ &=\Theta(t'-t)\int_t^{t'} \widetilde{f}_2(t',\tau) \widetilde{f}_1(\tau, t) \, \text{d}\tau,
\end{align*}
 which makes calculations involving such functions easier to carry out and shows that $\Sm$ is closed under  $\ast$-multiplication. 
 With the previous arguments, this demonstrates that $\D$ is closed under $\ast$-multiplication. As a consequence, given $f \in \D$ we can define the $k$th \emph{$\ast$-power} $f^{\ast k}$ as the $k$ $\ast$-products $f \ast f \ast \dots \ast f$ ($f^{\ast 0} = \delta(t'-t)$ by convention) \cite{GiscardPozza2019}.
 On $\Sm$ the $\ast$-product reduces to the so-called Volterra composition, a product between smooth functions of two-variables introduced by Volterra and Pérès \cite{Volterra1928}.

The $\ast$-product extends directly to distributions for which the smooth coefficients are \emph{matrices} of smooth coefficients by using the ordinary matrix product between the integrands in \eqref{eq:def:*}
 (see \cite{GiscardPozza2019} for more details). It is also well defined for distributions of $\D$ whose smooth coefficients depend on less than two variables. Indeed, consider a generalized function $f_3(t',t)=\widetilde{f}_3(t')\delta^{(i)}(t'-t)$ with $i\geq -1$ and $\delta^{(-1)}\equiv \Theta$. Then
\begin{align*}
\big(f_3 \ast f_1\big)(t',t)&= \widetilde{f}_3(t')\int_{-\infty}^{+\infty}  \delta^{(i)}(t'-\tau)f_1(\tau, t) \, \text{d}\tau,\\
\big(f_1\ast f_3\big)(t',t)&=\int_{-\infty}^{+\infty}  f_1(t',\tau)\widetilde{f}_3(\tau)\delta^{(i)}(\tau-t) \, \text{d}\tau.
\end{align*}
where $f_1(t',t)$ is defined as before.
Hence the variable of $\widetilde{f}_3(t')$ is treated as the left variable of a smooth function of two variables. This observation extends straightforwardly should $\widetilde{f}_3$ be constant and, by linearity, to any distribution of $\D$.

\subsection{Tridiagonalization: $\ast$-Lanczos algorithm}
 Let $\mathsf{A}(t',t):=\widetilde{\mathsf{A}}(t')\Theta(t'-t)$ with $\widetilde{\mathsf{A}}(t')$ a $N\times N$ time-dependent matrix.
 The $k$th matrix $\ast$-power of $\mathsf{A}$ is denoted by $\mathsf{A}^{*k}$.
As shown in \cite{Giscard2015}, if all entries $\widetilde{\mathsf{A}}(t')_{ij}$ are bounded over $I$, then the related time-ordered exponential $\mathsf{U}(t',t)$ can be expressed as
\begin{equation}\label{OrderedExp}
\mathsf{U}(t',t)=\Theta(t'-t) \ast \mathsf{R}_{\ast}(\mathsf{A})(t',t).
\end{equation}
Here $\mathsf{R}_{\ast}(\mathsf{A})$ is the $\ast$-resolvent of $\mathsf{A}$, defined as
\begin{equation}\label{RseriesA}
\mathsf{R}_{\ast}(\mathsf{A}):=\big(\mathsf{Id}\,1_\ast-\mathsf{A}\big)^{\ast-1}=\mathsf{Id}\,1_\ast+\sum_{k\geq 1}\mathsf{A}^{\ast k},
\end{equation}
the series on the right-hand side converging when $\widetilde{\mathsf{A}}$ elements are bounded.

\begin{table}
     \noindent\fbox{
\parbox{0.95\textwidth}{\begin{center}
\parbox{0.9\textwidth}{
\bigskip

 \noindent \underline{Input:} A complex time-dependent matrix $\mathsf{A}$, and
                  complex vectors $\b{w},\b{v}$ such that $\b{w}^H\b{v} = 1$.

 \noindent \underline{Output:} Coefficients $\alpha_0,\cdots,\, \alpha_{n-1}$ and $\beta_1,\cdots,\, \beta_{n-1}$ defining the matrix $\mathsf{T}_n$ of Eq.~(\ref{eq:tridiag}) which satisfies Eq.~(\ref{AjTjResult}).
  \begin{align*}
  & \textrm{Initialize: } \b{v}_{-1}=\b{w}_{-1}=0, \, \b{v}_0 = \b{v} \, 1_*, \, \b{w}_0^H = \b{w}^H 1_*. \\
  &    \alpha_0 = \b{w}^H \mathsf{A} \, \b{v}, \\
  &    \b{w}_1^H = \b{w}^H \mathsf{A} -  \alpha_{0}\, \b{w}^H,\\
  &    \b{\widehat v}_{1} = \mathsf{A}\, \b{v} - \b{v}\,\alpha_{0}, \\
  &    \beta_1 = \b{w}^H \mathsf{A}^{*2} \, \b{v} - \alpha_0^{*2}, \\
  &    \qquad\textrm{If } \beta_1 \textrm{ is not $*$-invertible, then stop, otherwise}, \\
  &     \b{v}_1 = \b{\widehat v}_1 * \beta_1^{*-1}, \\[1em]
  & \textrm{For } n=2,\dots \\ 
  &    \qquad \quad \alpha_{n-1} = \b{w}_{n-1}^H * \mathsf{A} * \b{v}_{n-1}, \\
  &    \qquad \quad \b{w}_{n}^H = \b{w}_{n-1}^H * \mathsf{A} -  \alpha_{n-1} * \b{w}_{n-1}^H - \beta_{n-1}*\b{w}_{n-2}^H, \\
  &    \qquad \quad \b{\widehat v}_{n} = \mathsf{A} * \b{v}_{n-1} - \b{v}_{n-1}*\alpha_{n-1} - \b{v}_{n-2}, \\
  &     \qquad \quad \beta_n = \b{w}_{n}^H * \mathsf{A} * \b{v}_{n-1}, \\
  &     \qquad \qquad\quad \textrm{If } \beta_n \textrm{ is not $*$-invertible, then stop, otherwise}, \\
  &     \qquad \quad  \b{v}_n = \b{\widehat v}_n * \beta_n^{*-1}, \\
  & \textrm{end}.
  \end{align*}
  }\end{center}}}
  \caption{The $*$-Lanczos Algorithm of \cite{GiscardPozza2019}.}\label{algo:*lan}
  \end{table} 
 
 ~
 
 Now we can recall results in \cite{GiscardPozza2019}: baring breakdowns--which we will characterize below--the $\ast$-Lanczos algorithm reproduced here in Table \ref{algo:*lan} produces a sequence of tridiagonal matrices $\mathsf{T}_n$, $1\leq n\leq N$, of the form
\begin{equation}\label{eq:tridiag}
    \mathsf{T}_n := \begin{bmatrix} 
            \alpha_0 & \delta &          & \\ 
            \beta_1  & \alpha_1 & \ddots & \\ 
                     & \ddots   & \ddots & \delta \\
                     &           & \beta_{n-1} & \alpha_{n-1}
          \end{bmatrix},
\end{equation}
and such that the matching $\ast$-moment property is achieved:
\begin{theorem}[\cite{GiscardPozza2019}]\label{thm:mmp}
   Let $\mathsf{A},\b{w},\b{v}$ and $\mathsf{T}_n$ be as described above, then
   \begin{equation}\label{AjTjResult}
        \b{w}^H (\mathsf{A}^{*j})\, \b{v} = \b{e}_1^H (\mathsf{T}_n^{*j})\,\b{e}_1, \quad \text{ for } \quad j=0,\dots, 2n-1.
   \end{equation}
\end{theorem}
Combining this with Eq.~(\ref{RseriesA}) we have, for $n=N$, the exact expression
$$
\b{w}^H\mathsf{U}(t',t)\b{v}=\Theta(t'-t) \ast \mathsf{R}_{\ast}(\mathsf{T}_{n})_{1,1}(t',t),
$$
while for $n<N$, the right-hand side yields an approximation to the time-ordered exponential. 
The method of path-sum \cite{Giscard2015} then gives explicitly
\begin{equation}\label{PSresult}
\mathsf{R}_{\ast}(\mathsf{T}_n)_{1,1}(t',t) 
= \Big(1_\ast - \alpha_0-\big(1_\ast-\alpha_1-(1_\ast-...)^{\ast-1}\ast\beta_2\big)^{\ast-1}\ast\beta_1\Big)^{\ast-1}.
\end{equation}
The $\ast$-Lanczos algorithm therefore provides the first general purpose approach to the calculation of ordered exponentials that is both exact, reaching the solution in a finite number of steps, and amenable to large-scale numerical computations.

\begin{remark}
The described tridiagonalization of the system of ODEs of \eqref{FundamentalSystem} can also be seen as a $\ast$-factorization of the matrix $\mathsf{A}$. Consider the matrices $\mathsf{W}_N = [\b{w}_0, \dots, \b{w}_{N-1}]$ and $\mathsf{V}_N = [\b{v}_0, \dots, \b{v}_{N-1}]$ composed of the vectors computed by the $\ast$-Lanczos algorithm. Then 
\begin{equation}\label{eq:asttrid}
    \mathsf{T}_N = \mathsf{W}_N^H \ast \mathsf{A} \ast \mathsf{V}_N, \quad \mathsf{W}_N^H \ast \mathsf{V}_N = \mathsf{Id}_{\ast},
\end{equation}
and
\begin{equation*}
   \mathsf{R}_{\ast}(\mathsf{A}) =  \mathsf{V}_N \ast  \mathsf{R}_{\ast}(\mathsf{T}_N)  \ast \mathsf{W}_N^H, 
\end{equation*}
with $\mathsf{Id}_{\ast}\equiv \mathsf{Id}\,1_\ast$ the identity with respect to the $\ast$-matrix-product \cite{GiscardPozza2019}.
\end{remark}

A crucial assumption underlying these results is that the algorithm suffers no breakdown. This is related to the nature of the $\alpha_j$ and $\beta_j$ distributions appearing in the $\mathsf{T}_n$ matrices and which are produced by the $\ast$-Lanczos procedure through recurrence relations. These necessitate the $\ast$-inversion of the $\beta_j$, i.e., the calculation of a distribution $\beta^{\ast-1}_j$ such that $\beta^{\ast-1}_j\ast \beta_j = \beta_j \ast \beta^{\ast-1}_j=1_\ast$. The paper \cite{GiscardPozza2019} assumed the existence of such $\ast$-inverses, without which the algorithm breaks down. If $\beta_j$ is not identically null, the existence of $\beta^{\ast-1}_j$ almost everywhere on $I$ was proven in a separate work \cite{ProceedingsPaper2019} assuming \textit{ad minima} that the $\alpha_j$ and $\beta_j$ would always be piecewise smooth elements of $\Sm$. In other terms, these works conjectured that the tridiagonalization of the system \eqref{FundamentalSystem} with $\mathsf{A}$ composed of functions of $\Sm$ is possible when the coefficients $\beta_1, \dots, \beta_{N-1}$ are not identically null. Here we establish this surprisingly difficult conjecture. 
Moreover we show that there exists $\b{w}, \b{v}$ so that the tridiagonalization \eqref{eq:asttrid} exists for $t',t \in I$ if for every $t' \in I$ it holds
\begin{equation}\label{eq:classtrid}
     \mathsf{J}_{t'} = \mathsf{Z}^{-1}_{t'}\, \mathsf{A}(t')\, \mathsf{Z}_{t'},
\end{equation}
with $\mathsf{J}_{t'}$ a tridiagonal matrix with nonzero off-diagonal elements, and $\mathsf{Z}_{t'}$ a square invertible matrix. This means that $\mathsf{A}(t')$ must be tridiagonalizable in the usual sense (note that \eqref{eq:classtrid} considers the usual matrix-product).

\section{Main Theorem: tridiagonalization with piecewise smooth functions and characterization of algorithmic breakdowns}
Before we state the main theorem on the tridiagonalization of systems of coupled linear differential equations with non-constant coefficients, we begin by exhibiting a relation between breakdowns in the $\ast$-Lanczos procedure and breakdowns in the ordinary non-Hermitian Lanczos procedure. This characterizes one of the assumptions of the main theorem and shows that the feasibility of tridiagonalization does not depend on the nature of the entries of the original matrix nor on the kind of product between these entries. Rather breakdowns in tridiagonalization must be topological in origin, i.e., they depend on the structure and the edge weights of the graph whose adjacency matrix is $\mathsf{A}$.
\begin{lemma}\label{lemma:breakdown}
   Let $\mathsf{T}_n$ be the tridiagonal matrix \eqref{eq:tridiag} obtained by $n$ iterations of the $\ast$-Lanczos algorithm in Table \ref{algo:*lan} with inputs $\mathsf{A}(t',t)=\tilde{\mathsf{A}}(t')\Theta(t'-t), \b{w}, \b{v}$,
   where all the entries of $\widetilde{\mathsf{A}}(t')$ are smooth functions of $t'$, and $\b{w},\b{v}$ are time-independent vectors with $\b{w}^H\b{v}=1$.
   Assume that the $\ast$-Lanczos coefficients $\alpha_{j-1}, \beta_j$ are in $\Sm$ and that $\beta_j(t,t) \equiv 0$, for every $j=1,\dots,n-1$.
   Let us denote with $\widetilde{\beta}_j^{(1,0)}(t',t)$ and $\widetilde{\beta}_j^{(0,1)}(t',t)$ respectively the derivative with respect to $t'$ and $t$ of $\widetilde{\beta}_j(t',t)$.
   Then the following statements are equivalent:
   \begin{enumerate}
       \item $\widetilde{\beta}_1^{(1,0)}(t,t), \dots, \widetilde{\beta}_{n-1}^{(1,0)}(t,t)$ are not identically null on $I$;\label{stat:beta1}
        \item $\widetilde{\beta}_1^{(0,1)}(t,t), \dots, \widetilde{\beta}_{n-1}^{(0,1)}(t,t)$ are not identically null on $I$;\label{stat:beta2} 
       \item There exists at least one $\rho \in I$ so that the usual non-Hermitian Lanczos algorithm with inputs $\widetilde{\mathsf{A}}(\rho), \b{w}, \b{v}$ has no (true) breakdown in the first $n-1$ iterations.\label{stat:lanc}
   \end{enumerate}
\end{lemma}
Note that statement \ref{stat:beta1} (or equivalently Statement \ref{stat:beta2}) in Lemma~\ref{lemma:breakdown} also implies that there cannot be a breakdown in the first $n$ iterations of the $\ast$-Lanczos algorithm in Table \ref{algo:*lan}, meaning that $\beta_1, \dots, \beta_{n-1}$ are $\ast$-invertible almost everywhere on $I$.
Hence Statement \ref{stat:lanc} in Lemma~\ref{lemma:breakdown} is a sufficient condition for not having a breakdown in the $\ast$-Lanczos Algorithm.
We also remark that the matrix $\widetilde{\mathsf{A}}(\rho)$ is tridiagonalizable in the sense of \eqref{eq:classtrid} if and only if there exists $\b{w}, \b{v}$
so that the usual non-Hermitian Lanczos algorithm with inputs $\widetilde{\mathsf{A}}(\rho), \b{w}, \b{v}$  has no (true) breakdown until the last iteration; see, e.g., \cite{Par92}.
\\[-.5em] 

Now we are ready to state our main result.
\begin{theorem}\label{MainTheorem}
Let $\widetilde{\mathsf{A}}(t')$ be a $N\times N$ time dependent matrix and let $\mathsf{U}(t',t)$ be its time-ordered exponential. Let $\b{w}$ and $\b{v}$ be time-independent $N\times 1$ vectors with $\b{w}^H\b{v}=1$.
Assume that for every $t'$ in $I$, the usual non-Hermitian Lanczos algorithm with inputs $\widetilde{\mathsf{A}}(t'), \b{w}, \b{v}$ has no  (true) breakdown in the $k$th iteration, for $k=1,\dots,N-1$.
If all the entries of $\widetilde{\mathsf{A}}(t')$ are smooth functions of $t'$,
then there are smooth functions $\widetilde{\alpha}_{0\leq j\leq N-1}$, $\widetilde{\beta}_{1\leq i\leq N-1}$ and distributions
\begin{align*}
&\alpha_{j}(t',t):=\widetilde{\alpha}_j(t',t)\Theta(t'-t),\\
&\beta_{i}(t',t):=\widetilde{\beta}_{i}(t',t)\Theta(t'-t),
\end{align*}
such that $\beta_i(t,t)\equiv 0$, $\widetilde{\beta}^{(1,0)}_i(t,t)\not \equiv 0$, $\widetilde{\beta}^{(0,1)}_i(t,t)\not \equiv 0$ 
and the tridiagonal matrix 
\begin{equation*}
    \mathsf{T} := \begin{bmatrix} 
            \alpha_0 & \delta &          & \\ 
            \beta_1  & \alpha_1 & \ddots & \\ 
                     & \ddots   & \ddots & \delta \\
                     &           & \beta_{N-1} & \alpha_{N-1}
          \end{bmatrix}, 
\end{equation*}
obeys 
\begin{align*}
&\b{w}^H\mathsf{A}^{\ast n}\b{v}=\big(\mathsf{T}^{\ast n}\big)_{1,1},~n\geq 0\\
&\b{w}^H\mathsf{U}(t',t)\b{v}=\Theta(t'-t)\ast \mathsf{R}_{\ast}(\mathsf{T})_{1,1}(t',t),
\end{align*}
where $\mathsf{A}(t',t):=\widetilde{\mathsf{A}}(t')\Theta(t'-t)$. Furthermore, the $\ast$-inverses $\beta_{1\leq i\leq N-1}^{\ast -1}$ exist and are of the form $\beta_i^{\ast -1}=\delta^{(3)}\ast b$, with $b\in \Sm$.
\end{theorem}

\begin{remark}\label{Rem:nonessential}
 The Dirac delta distributions $\delta$ in the upper diagonal of the matrix $\mathsf{T}$ are non-essential. Indeed, one can instead choose to replace them by the piecewise smooth function $\Theta(t'-t)$, $\Theta(0)=1$ if at the same time all $\beta_j$ coefficients are replaced with $\frac{\partial\tilde{\beta}_j}{\partial t'}\times \Theta(t'-t)$ for $1\leq j\leq N-1$. The feasibility of this operation is guaranteed by the main theorem above. 
 Here, we retain the version with isolated non-essential delta distributions for the ease of the proof. 
\end{remark}

The proofs of the Theorem~\ref{MainTheorem} and Lemma~\ref{lemma:breakdown} occupy the remainder of the present work. We proceed as follows: in Section~\ref{StarAction} we begin with basic results pertaining to the $\ast$-action of derivatives of the Dirac delta distribution. In Section~\ref{TechnicalLemmas} we gather technical Lemmas pertaining to $\ast$-products of piecewise smooth functions of $\Sm$ as well as on the existence and form of their $\ast$-inverses. Section \ref{sec:breakdown} proves Lemma~\ref{lemma:breakdown}. The previous results lead onto the main argument of the proof, in Section~\ref{MainArg}, which is an induction on the $\alpha_j$ and $\beta_j$ generalized functions produced by the $\ast$-Lanczos algorithm.\\

\section{Proofs}
\subsection{$\ast$-Action of delta derivatives}\label{StarAction}
We begin by recalling basic results pertaining to the $\ast$-action of derivatives of the Dirac delta distribution. We denote by $\delta^{(j)}(t'-t)$ the $j$th derivative of the Dirac delta distribution $\delta(t'-t)\equiv \delta^{(0)}(t'-t)$. We generally omit the $(t'-t)$ argument to alleviate the equations, unless absolutely necessary (we do the same with the Heaviside function $\Theta(t'-t)$). For a distribution $f$ depending on two times or less, we have \cite{ProceedingsPaper2019,schwartz1978}
\begin{align*}
    \big(f\ast\delta^{(j)}\big)(t',t)  &=(-1)^jf^{(0,j)}(t',t),\\
    \big(\delta^{(j)}\ast f\big)(t',t) &=f^{(j,0)}(t',t),\\
    \delta^{(j)}\ast\delta^{(k)}&=\delta^{(j+k)}, \\
    \Theta \ast \delta' &= \delta. 
\end{align*}
The notation $f^{(j,k)}(\tau,\rho)$ stands for the $j$th $t'$-derivative and $k$th $t$-derivative of $f$ evaluated at $t'=\tau, t=\rho$ with the understanding that $j=0$ or $k=0$ means no derivative is taken. 
Since the $\ast$-product is associative, $\big(\delta^{(i)}\ast f\big)\ast \delta^{(j)}=\delta^{(i)}\ast \big(f \ast \delta^{(j)}\big)$ and the notation $f^{(i,j)}$ is well defined.
For piecewise smooth functions $f\in \Sm$,  $f(t',t)=\widetilde{f}(t',t)\Theta(t'-t)$, this implies
\begin{subequations}\label{eq:deltader} \begin{align}
       \delta^{(j)} * f(t',t) &= \widetilde{f}^{(j,0)}(t',t) \Theta +\sum_{k=0}^{j-1} \widetilde{f}^{(j-k-1,0)}(t,t) \delta^{(k)},\\
      f(t',t)\ast  \delta^{(j)}  &=(-1)^j \widetilde{f}^{(0,j)}(t',t) \Theta + \sum_{k=0}^{j-1}(-1)^{k+j+1}\widetilde{f}^{(0,j-k-1)}(t',t') \delta^{(k)}.
   \end{align}
\end{subequations}

Finally, we note the following identities between distributions for $j\geq 0$
\begin{subequations}\label{RelDeltas}
\begin{align}
\widetilde{f}(t')\delta^{(j)}(t'-t)&=(-1)^j \big(\widetilde{f}(t)\delta(t'-t)\big)^{(0,j)},\\
\widetilde{f}(t)\delta^{(j)}(t'-t)&= \big(\widetilde{f}(t')\delta(t'-t)\big)^{(j,0)};
\end{align}
\end{subequations}
see, e.g., \cite{ProceedingsPaper2019,schwartz1978}.


\subsection{Technical results}\label{TechnicalLemmas}
In this section we gather crucial technical Lemmas. The first result pertains to derivatives of $\ast$-products of functions of $\Sm$, establishing that all derivatives of order $k\leq n$ of a $\ast$-product of $n$ functions are identically null when $t'=t$. The second Lemma gives the generic form for the $\ast$-inverse of a function of $\Sm$.

\begin{lemma}\label{ZerosDelta}
Let $f_j(t',t):=\widetilde{f}_j(t',t)\Theta(t'-t)$, $j=1,2,\dots$, be a family of functions of $\Sm$.
Let $F_n:=\widetilde{F}_n(t',t)\Theta(t'-t)=f_n\ast \cdots \ast f_1$ for $n\geq 2$. Then for $0\leq q+r\leq n-2$ we have,
\begin{equation}\label{DiagDeriv}
\widetilde{F}_n^{(q,r)}(t,t) \equiv 0,
\end{equation}
and consequently, for $0\leq q+r\leq n-1$,
\begin{equation}\label{DeltaAction}
\big(\delta^{(q)}\ast F_n\ast \delta^{(r)}\big)(t',t)=(-1)^r\widetilde{F}_n^{(q,r)}(t',t)\,\Theta(t'-t).
\end{equation}
In particular, $\widetilde{F}_n^{(n-1,0)}(t,t)=\widetilde{f}_n(t,t)\cdots \widetilde{f}_2(t,t)\widetilde{f}_1(t,t)=(-1)^{n-1}\widetilde{F}_n^{(0,n-1)}(t,t)$. More generally, if none of the $\widetilde{f}_j(t,t)$ are identically null, then $\widetilde{F}_n^{(q,r)}(t,t)\not\equiv 0$ when $q+r=n-1$.\\[-.5em]

If $\widetilde{f}_j(t,t)\equiv0$ is true for all $1\leq j\leq n$, then Eq.~(\ref{DiagDeriv}) is true for $n\geq 1$ as long as $0\leq q+r\leq 2n-2$ and  Eq.~(\ref{DeltaAction}) holds whenever $n\geq 1$ and $0\leq q+r\leq 2n-1$. If in addition none of the $\widetilde{f}_j(t,t)^{(1,0)}$ are identically null, then $\widetilde{F}_n^{(q,r)}(t,t)\not\equiv 0$ when $q+r=2n-1$.
\end{lemma}

This Lemma extends naturally to $\ast$-products of functions of $\Sm$ whose smooth part depends on less than two time variables, e.g. $\widetilde{a}(t')\Theta(t'-t)$. 

\begin{proof}
We proceed by induction on $n$.
The base case, at $n=2$, follows from a direct calculation
$$
F_2(t',t)=\big(f_1\ast f_1\big)(t',t)=\int_{t}^{t'}\widetilde{f}_2(t',\tau)\widetilde{f}_1(\tau,t)d\tau~\Theta(t'-t).
$$
Since both $\widetilde{f}_2$ and $\widetilde{f}_1$ are continuous as functions over $I^2$, then the above integral vanishes under the limit $t'\to t$, establishing that $F_2(t,t)\equiv 0$. For the derivatives of $F_2$, we get that $\widetilde{F}_2^{(1,0)}(t,t)=\widetilde{f}_1(t,t)\widetilde{f}_2(t,t)=-\widetilde{F}_2^{(0,1)}(t,t)$ is not identically null whenever neither $\widetilde{f}_1(t,t)$ nor $\widetilde{f}_2(t,t)$ are identically null.\\[-.5em]


Now, assume that the Lemma holds for every $\ast$-product of $n$ functions in $\Sm$ and let $F_n(t',t):=\widetilde{F}_n(t',t) \Theta(t'-t)=f_n\ast f_{n-1}\ast\cdots \ast f_1$ with $n\geq 2$. 
We will establish the Lemma by proving that this implies the required properties for $F_{n+1}$.

We get
$ \delta^{(q)}\ast F_{n+1} \ast \delta^{(r)} = \delta^{(q)}\ast f_{n+1}\ast F_n \ast \delta^{(r)}$.
For $0 \leq q+r \leq n-1$, by the inductive assumption and Eq.~\eqref{eq:deltader} we get
$$ \delta^{(q)}\ast F_n \ast \delta^{(r)} = (-1)^r\widetilde{F}_n^{(q,r)}\Theta.$$
Thus
$$ F_n = \Theta^{*q} \ast (-1)^r\widetilde{F}_n^{(q,r)}\Theta \ast \Theta^{*r}.$$
Therefore
$$ \delta^{(q)}\ast f_{n+1}\ast F_n \ast \delta^{(r)} = \delta^{(q)}\ast \left(f_{n+1} \ast \Theta^{*q}\right) \ast (-1)^r\widetilde{F}_n^{(q,r)}\Theta. $$
Since $q+1 \leq n$, the Lemma holds for $f_{n+1} \ast \Theta^{*q}$, giving $\delta^{(q)}\ast \left(f_{n+1} \ast \Theta^{*q}\right) = G_q \in \Sm$.
Finally we get
$$ \delta^{(q)}\ast f_{n+1}\ast F_n \ast \delta^{(r)} = \widetilde{G}_q\Theta \ast (-1)^r\widetilde{F}_n^{(q,r)}\Theta, $$
which is a $\ast$-product of two functions in $\Sm$.
Hence, as we have already proved in the base case, 
$$ (-1)^r\widetilde{F}_{n+1}^{(q,r)}(t,t) = \big(\delta^{(q)}\ast f_{n+1}\ast F_n \ast \delta^{(r)}\big)\big|_{t'=t} \equiv 0, \quad 0 \leq q+r \leq n-1,$$
from which we get
\begin{equation*}
\big(\delta^{(q)}\ast F_{n+1}\ast \delta^{(r)}\big)(t',t)=(-1)^r\widetilde{F}_{n+1}^{(q,r)}(t',t)\,\Theta(t'-t), \quad 0 \leq q+r \leq n.
\end{equation*}
\smallskip

There remains to establish that $\widetilde{F}_{n+1}^{(q,r)}(t,t)$ is not identically null for $q+r=n$ if none of the $f_n(t,t)\equiv 0$. This follows from the observation that since 
 $G_q(t,t) = \widetilde{f}_{n+1}(t,t)$ for every $q \geq 0$, then given $q+r=n$
$$ \widetilde{F}_{n+1}^{(q,r)}(t,t)= \left( \delta'\ast  \widetilde{G}_{q-1}\Theta \ast (-1)^r\widetilde{F}_n^{(q-1,r)}\Theta \right)\Big|_{t'=t} = (-1)^r \widetilde{f}_{n+1}(t,t) \widetilde{F}_n^{(q,r)}(t,t), $$
and, similarly, we get
$$ \widetilde{F}_{n+1}^{(q,r)}(t',t')= \left( \widetilde{G}_q\Theta \ast (-1)^{r-1}\widetilde{F}_n^{(q,r-1)}\Theta \ast \delta' \right)\Big|_{t=t'} = (-1)^{r} \widetilde{f}_{n+1}(t',t') \widetilde{F}_n^{(q,r)}(t',t'). $$
Hence, for $q+r=n$, $\widetilde{F}_{n+1}^{(q,r)}(t',t')$ is not identically null since $f_{n+1}(t,t)\not \equiv0$ and $\widetilde{F}_{n}^{(q,r-1)}(t,t)\not \equiv 0$ by assumption and induction, respectively. The same argument establishes that $\widetilde{F}_{n+1}^{(n+1,0)}(t,t)=\prod_{j=1}^{n+1}\widetilde{f}_j(t,t)$.
This gives the first part of the Lemma.\\[-.5em]

Assuming $f_j(t,t)$ identically null for $j=1,\dots,n+1$, by Eq.~\eqref{eq:deltader} we get 
$$f_j(t',t) = \Theta(t'-t) \ast \widetilde{f}_j^{(1,0)}(t',t), \quad j=1,\dots,n+1.$$
Hence $F_{n+1} = \Theta \ast f_1 \ast \Theta \ast f_2 \ast \cdots \ast \Theta \ast f_{n+1}$ is a $\ast$-product of $2n+2$ functions in $\Sm$. 
Applying the first part of the Lemma to such a $\ast$-product, we conclude the proof.
\end{proof}

\smallskip

\begin{lemma}\label{InverseFormF}
Let $f(t',t):=\widetilde{f}(t',t)\Theta(t'-t)$ with $\widetilde{f}$ smooth. Let $k\in\mathbb{N}$ be the smallest integer such that $\widetilde{f}^{(k,0)}(t,t)$  and $\widetilde{f}^{(0,k)}(t,t)$ are not identically null. Then the $\ast$-inverse of $f$ exists almost everywhere on $I^2$ and is given by 
\begin{align*}
f^{\ast-1} =f^L_{\text{inv}}\ast \delta^{(k+2)}=\delta^{(k+2)}\ast f_{inv}^R.
\end{align*}
Both $\widetilde{f}^{L,R}_{inv}(t,t)$ are not identically null and
\begin{align*}
f^L_{\text{inv}}&:=\big(\widetilde{f}^{(k,0)}(t',t')\big)^{-1}\Theta+\int_{t}^{t'}\widetilde{r}^L(t',\tau)d\tau~\Theta,\\
r^L(t',t)&:=\widetilde{r}^L(t',t)\Theta=\frac{1}{\widetilde{f}^{(k,0)}(t,t)}\sum_{n=1}^\infty (-g^L)^{\ast n},\\
g^L(t',t)&:=\sum_{m=1}^{\infty}\widetilde{f}^{(m+k,0)}(t,t) \Theta^{\ast m}.
\end{align*}
while
\begin{align*}
f^R_{\text{inv}}&:=\big(\widetilde{f}^{(0,k)}(t,t)\big)^{-1}\Theta+\int_{t}^{t'}\widetilde{r}^R(\tau,t)d\tau~\Theta,\\
r^R(t',t)&:=\widetilde{r}^R(t',t)\Theta=\frac{1}{\widetilde{f}^{(0,k)}(t',t')}\sum_{n=1}^\infty (-g^R)^{\ast n},\\
g^R(t',t)&:=\sum_{m=1}^{\infty}\widetilde{f}^{(0,m+k)}(t',t')(-1)^{m+k} \Theta^{\ast m}.
\end{align*}
In addition, if $\widetilde{f}^{(k,0)}(t,t)$ and $\widetilde{f}^{(0,k)}(t',t')$ are nonzero on $I$, then $f^{L,R}_\text{inv}\in \Sm(I)$.
\end{lemma}

\begin{proof}
Because $\widetilde{f}$ is smooth in both time variables, we can appeal to the Taylor series representation 
\begin{align*}
f(t',t)&=\sum_{n\geq k}\widetilde{f}^{(n,0)}(t,t) \frac{(t'-t)^n}{n!}\Theta,\\
&=\sum_{n\geq k}\widetilde{f}^{(n,0)}(t,t) \Theta^{\ast n+1},\\
&=\Theta^{\ast k+1}\ast \left(\widetilde{f}^{(k,0)}(t,t)\delta+\sum_{m=1}^{\infty}\widetilde{f}^{(m+k,0)}(t,t) \Theta^{\ast m} \right),\\
&=\Theta^{\ast k+1}\ast \left(\widetilde{f}^{(k,0)}(t,t)\delta+\widetilde{g}(t',t)\Theta\right),
\end{align*}
with 
$$
g(t',t):=\sum_{m=1}^{\infty}\widetilde{f}^{(m+k,0)}(t,t) \Theta^{\ast m}=\sum_{m=1} \widetilde{f}^{(m+k,0)}(t,t) \frac{(t'-t)^{m-1}}{(m-1)!}\,\Theta.
$$
The inverse of $f(t',t)$ will therefore be of the form $ \left(\widetilde{f}^{(k,0)}(t,t)\delta+g(t',t)\right)^{\ast-1}\ast \delta^{(k+1)}$, provided the inverse of $\widetilde{f}^{(k,0)}(t,t)\delta+\widetilde{g}(t',t)\Theta$ does indeed exist. 

In order to alleviate the notation, let $\widetilde{f}_t$ designate $\widetilde{f}^{(k,0)}(t,t)$. Let us suppose that the $\ast$-inverse $r$ of $\widetilde{f}_t \delta+g$ exists. Then it should satisfy
$
(\widetilde{f}_t \delta + g)\ast r =\delta.$
Expanding this out with the help of Eq.(\ref{RelDeltas}), we get $\delta- g\ast r =f_{t'} r$, that is $f_{t'}^{-1}\delta - f_{t'}^{-1} (g\ast r)=r$. Iteratively replacing $r$ on the left-hand side by its value as given by the right-hand side we get,
\begin{align*}
r=\widetilde{f}^{-1}_{t'}\delta - g\ast (\widetilde{f}^{-1}_{t'}\delta)+ g\ast g \ast (\widetilde{f}^{-1}_{t'}\delta) - \cdots &= \sum_{n=0}^\infty (-g)^{\ast n} \ast (\widetilde{f}^{-1}_{t'}\delta),\\
&=\widetilde{f}^{-1}_t\sum_{n=0}^\infty (-g)^{\ast n}
\end{align*}
Given the form of $g$ and $\widetilde{f}$ being smooth, $\widetilde{g}$ is bounded and the series above is convergent, except possibly at a countably finite number of points $t\in I$ for which $\widetilde{f}^{(k,0)}(t,t)=0$.
Therefore $r$ exists with,
$$
r=f_t^{-1} \delta + \widetilde{r}(t',t)\Theta,~\text{ with }~\widetilde{r}(t',t)\Theta:=\widetilde{f}^{-1}_t\sum_{n=1}^\infty (-g)^{\ast n}.
$$

Returning to the $\ast$-inverse of $f$ we have thus proven that it exists and takes on the form, 
$$
f^{\ast-1}= \left(\frac{1}{\widetilde{f}_t}\delta + \widetilde{r}(t',t)\Theta\right)\ast \delta^{(k+1)},
$$
where $\widetilde{r}$ is an ordinary function. Now it suffices to observe that 
\begin{align*}
\frac{1}{\widetilde{f}^{(k,0)}(t',t')}\Theta\ast \delta'\equiv\widetilde{f}_{t'}^{-1}\Theta\ast \delta' =\widetilde{f}_{t'}^{-1} \delta=\widetilde{f}_{t}^{-1} \delta,
\end{align*}
where the last equality follows from Eq.(\ref{RelDeltas}) with $j=0$. Furthermore
$\int_{t}^{t'}\widetilde{r}(t',\tau)d\tau~\Theta \ast \delta'=\widetilde{r}(t',t)\Theta$. Therefore, 
\begin{equation}\label{finvForm}
f^{\ast-1}=\left(\widetilde{f}_{t'}^{-1}\Theta+\int_{t}^{t'}\widetilde{r}(t',\tau)d\tau~\Theta\right)\ast \delta^{(k+2)},
\end{equation}
and the content of the parenthesis is $f^L_{inv}$.
There remains to show that $f_{inv}(t,t)$ is not identically null. To this end, remark that as $\widetilde{r}$ is smooth in both time variables, the integral from $t$ to $t'$ of $\widetilde{r}(t',\tau)$ vanishes under the limit $t'\to t$. Given that here $\Theta(0)=1$, there remains 
$$
f^L_{inv}(t,t)=\widetilde{f}_t^{-1}.
$$
which is not identically null, by assumption. The proof for $f_{inv}^R$ is entirely similar, with the starting Taylor expansion being around $t',t'$ instead of $t,t$. This establishes the Lemma. 
\end{proof}
~\\

\subsection{Proof of Lemma~\ref{lemma:breakdown}}\label{sec:breakdown}
Note that $\beta_j = \Theta \ast \widetilde{\beta}_j^{(1,0)}\Theta$ since $\beta_j(t,t)\equiv 0$, for $j=1,\dots, n-1$.
Considering that $\big(\mathsf{T}_{n}^{\ast j+1}\big)(t',t)_{11}$ can be written as a sum of $\ast$-products of $j+1$ functions among $\alpha_0,\dots, \alpha_{n-1}$, $\beta_1^{(1,0)},\dots, \beta_{n-1}^{(1,0)}$ and $\Theta$, Lemma~\ref{ZerosDelta} gives
\begin{equation}\label{eq:derT}
    \frac{\partial^{j}}{\partial t^{j}}\big(\mathsf{T}_{n}^{\ast j+1}\big)(t',t)_{11}\big|_{t'=t} = \left(\mathsf{J}_{n}^{j+1}\right)(t,t)_{11}, \quad j = 0, 1, 2 \dots \, , 
\end{equation}
with
\begin{equation*}
    \mathsf{J}_{n}(t) := \begin{bmatrix} 
            \widetilde{\alpha}_0(t,t) & 1 &          & \\ 
            \widetilde{\beta}_1^{(1,0)}(t,t)  & \widetilde{\alpha}_1(t,t) & \ddots & \\ 
                     & \ddots   & \ddots & 1 \\
                     &           & \widetilde{\beta}_{n-1}^{(1,0)}(t,t) & \widetilde{\alpha}_{n-1}(t,t)
          \end{bmatrix}.
\end{equation*}
Moreover, $\b{w}^H \mathsf{A}^{\ast j+1} \b{v}$ is a sum of $\ast$-products of $j+1$ functions in $\Sm$. Hence by Lemma~\ref{ZerosDelta}
\begin{equation}\label{eq:derA}
    \frac{\partial^{j}}{\partial t^{j}}\big(\b{w}^H \mathsf{A}^{\ast j+1} \b{v}\big)(t',t)\big|_{t'=t} = \b{w}^H \widetilde{\mathsf{A}}^{j}(t)\b{v}, \quad j = 0, 1, 2 \dots \, ,
\end{equation}
Here, notice that ordinary matrix powers appear on the right hand side and not $\ast$-powers anymore.
Then, Theorem~\ref{thm:mmp} implies
\begin{equation}\label{eq:usualLan}
\b{w}^H \widetilde{\mathsf{A}}^{j}(t)\b{v} = \mathsf{J}_{n}^{j}(t)_{11}, \quad j=0,\dots, 2n-1.
\end{equation}
Let us fix $\rho \in I$. The following statements are equivalent (see, e.g., \cite{Par92,PozPraStr18,PozPra19}):
\begin{itemize}
    \item The (usual) non-Hermitian Lanczos process with inputs $\widetilde{\mathsf{A}}(\rho)$, $\b{w}, \b{v}$ generates an $n \times n$ (time-independent) tridiagonal matrix $\mathsf{S}_{n,\rho}$ with nonzero elements on the off-diagonals;
    \item The (usual) non-Hermitian Lanczos process with inputs $\widetilde{\mathsf{A}}(\rho)$, $\b{w}, \b{v}$ does not have a (true) breakdown in the first $n-1$ iterations;
    \item There exists a $n \times n$ (time-independent) tridiagonal matrix with nonzero elements on the off-diagonal ${\mathsf{H}}_{n,\rho}$ satisfying
 \begin{equation*}
\b{w}^H \left(\widetilde{\mathsf{A}}(\rho)\right)^{j}\b{v} = \b{e}_1^T\, \left({\mathsf{H}}_{n,\rho}\right)^{j}\, \b{e}_1, \quad j=0,\dots, 2n-1.
\end{equation*}
\end{itemize}
In particular, every such ${\mathsf{H}}_{n,\rho}$ is in the form 
$${\mathsf{H}}_{n,\rho} = {\mathsf{D}}_\rho \, {\mathsf{S}}_{n,\rho} \, {\mathsf{D}}_{\rho}^{-1},$$ 
with ${\mathsf{D}}_\rho$ a non-singular diagonal matrix. 
Therefore if for a fixed $\rho \in I$ the coefficients $\widetilde{\beta}_1^{(1,0)}(\rho,\rho), \dots, \widetilde{\beta}_{n-1}^{(1,0)}(\rho,\rho)$ are nonzero, then by Eq.~\eqref{eq:usualLan} the non-Hermitian Lanczos process on $\widetilde{\mathsf{A}}(\rho), \b{w}, \b{v}$ does not have a (true) breakdown in the first $n-1$ iterations and give as an output a tridiagonal matrix $\mathsf{S}_{n,\rho}$ so that 
$$\mathsf{J}_n(\rho) = {\mathsf{D}}_\rho \, \mathsf{S}_{n,\rho} {\mathsf{D}}_\rho^{-1},$$
with ${\mathsf{D}}_\rho$ a nonsingular diagonal matrix.
 Conversely, if for a fixed $\rho \in I$ the non-Hermitian Lanczos process on $\widetilde{\mathsf{A}}(\rho), \b{w}, \b{v}$ has not a (true) breakdown in the first $n-1$ iterations, then it generates a tridiagonal matrix ${\mathsf{S}}_{n,\rho}$ with nonzero elements in the off-diagonal. Therefore since ${\mathsf{J}}_n(\rho) = {\mathsf{D}}_\rho \, {\mathsf{S}}_{n,\rho} {\mathsf{D}}^{-1}_\rho$ with ${\mathsf{D}}_\rho$ a non-singular diagonal matrix, the coefficients $\widetilde{\beta}_1^{(1,0)}(\rho,\rho), \dots, \widetilde{\beta}_{n-1}^{(1,0)}(\rho,\rho)$ are nonzero.
  Being $\widetilde{\beta}_1^{(1,0)}(t,t), \dots, \widetilde{\beta}_{n-1}^{(1,0)}(t,t)$ smooth functions of $t \in I$, they are either identically null on $I$ or nonzero almost everywhere on $I$, showing that Statement \ref{stat:beta1} is equivalent to Statement \ref{stat:lanc}.
 
 By similar arguments, Statement \ref{stat:beta2} is equivalent to Statement \ref{stat:lanc}, concluding the proof.

\subsection{Proof of Theorem~\ref{MainTheorem}}\label{MainArg}
We are now ready to prove Theorem~\ref{MainTheorem}. We begin with proving the Theorem's statements concerning the off-diagonal coefficients $\beta_j$. For all integers $1\leq n\leq N-1$, we denote $\mathfrak{P}_n$  the proposition:\\[-.7em]

$\mathfrak{P}_n:=$``For all $1\leq j\leq n$, $\beta_j\in \Sm$, $\beta_j(t',t')\equiv 0$ is identically null and neither $\beta_j^{(1,0)}(t',t')$ nor $\beta_j^{(0,1)}(t,t)$ are identically null.''\\[-.7em] 

We establish $\mathfrak{P}_n$ by induction.

\begin{proof}[Proof for the coefficients $\beta$]                   
Observe that the $j$th $\ast$-moment of the matrix $\mathsf{A}$ satisfies $m_j(t',t) := \b{w}^H \mathsf{A}^{\ast j} \b{v}\in \Sm$ for $j\in\mathbb{N}$. Since by definition $\alpha_0(t',t) = \b{w}^H \mathsf{A} \b{v}$, $\alpha_0\in \Sm$ and 
\begin{equation*}
    \beta_1(t',t) =  m_2(t',t) - \alpha_0^{\ast 2}(t',t),
\end{equation*}
then $\beta_1\in \Sm$.
In addition, the $\ast$-product of two elements of $\Sm$ is null whenever $t'=t$ owing to the continuity of the integrand, and thus we immediately get $m_2(t,t)=\alpha_0^{\ast 2}(t,t)\equiv 0$ entailing that $\beta_{1}(t,t)\equiv 0$. 
Finally, we get $\widetilde{\beta}_1^{(1,0)}(t',t'), \widetilde{\beta}_1^{(0,1)}(t',t')$ not identically null by Lemma~\ref{lemma:breakdown}.



Assuming $n \geq 1$, the central object of interest is \begin{equation}\label{FnDef}
F_{n+1}(t',t):=m_{2n+2}(t',t)-\big(\mathsf{T}_{n+1}^{\ast 2n+2}\big)(t',t)_{11}.
\end{equation}
 Observe that $m_{2n+2}(t',t)$ is a sum of  $\ast$-products of $2n+2$ functions in $\Sm$.
 Moreover, by the inductive assumption, for $j=1,\dots,n$ we have $\beta_j \in \Sm$ and $\beta_j(t,t)\equiv 0$.
As a consequence, $\beta_j(t',t) = - \widetilde{\beta}_j^{(0,1)}(t',t)\Theta \ast \Theta = \Theta \ast \widetilde{\beta}_j^{(1,0)}(t',t)\Theta$.
Hence $\big(\mathsf{T}_{n+1}^{\ast 2n+2}\big)(t',t)_{11}$ can be written as a sum of $\ast$-products of $2n+2$ functions among $\alpha_0,\dots, \alpha_n$, $\widetilde{\beta}_1^{(0,1)}\Theta,\dots, \widetilde{\beta}_n^{(0,1)}\Theta$ and $\Theta$. Then $F_{n+1}\in\Sm$ and, by Lemma~\ref{ZerosDelta}, for $q+r\leq 2n$,
\begin{equation}\label{Fn1Deriv}
\widetilde{F}_{n+1}^{(q,r)}(t,t)\equiv 0.
\end{equation}
while for $0\leq q+r\leq 2n+1$,
\begin{equation}\label{Fn1Result}
\big(\delta^{(q)}\ast F_{n+1}\ast \delta^{(r)}\big)(t',t)=(-1)^r\widetilde{F}_{n+1}^{(q,r)}(t',t)\,\Theta(t'-t).
\end{equation}
We can further identify $F_{n+1}$ upon noting that since $m_{2n+2}(t',t)=\big(\mathsf{T}_{k}^{\ast 2n+2}\big)_{11}$ whenever $k\geq n+2$. Since
 \begin{equation}\label{BetaProduct}
 \big(\mathsf{T}_{k}^{\ast 2n+2}\big)_{11}-\big(\mathsf{T}_{n+1}^{\ast 2n+2}\big)_{11}=\beta_{n+1}\ast\cdots\ast\beta_2\ast\beta_1,
\end{equation}
we get
 \begin{equation}\label{FormFnBeta}
 F_{n+1}=\beta_{n+1}\ast\cdots\ast\beta_2\ast\beta_1. 
 \end{equation}
 From now on, we suppose that $F_{n+1}(t',t)$ is not identically null over $I^2$. Indeed, should it be the case, then Eq.~(\ref{FormFnBeta}) implies that $\beta_{n+1}(t',t)$ is identically null, which corresponds to a breakdown of the $\ast$-Lanczos algorithm. Lemma~\ref{lemma:breakdown} shows that such a case is in contradiction with the theorem assumptions and in fact corresponds to a breakdown of the ordinary non-Hermitian Lanczos procedure.
 
 In order to determine what kind of distribution is $\beta_{n+1}$, we seek to express it as $\beta_{n+1}= F_{n+1}\ast  F_{n}^{\ast-1}$, where $F_n:=\beta_{n}\ast\cdots\ast\beta_2\ast\beta_1$. To this end, we first need to show the existence of $F_{n}^{\ast-1}$ and precisely control what form this may possibly take.
 We exploit again the fact that $\beta_j(t',t) = \Theta \ast \widetilde{\beta}_j^{(1,0)}(t',t)\Theta$ getting
$$ F_{n}=\Theta \ast \beta_{n}^{(1,0)}\ast\cdots\ast \Theta \ast\beta_2^{(1,0)}\ast \Theta \ast \beta_1^{(1,0)}.$$
Considering that by induction $\beta_j^{(1,0)}(t,t) \not\equiv 0$ for $j=1,\dots,2n-2$, Lemma~\ref{ZerosDelta} gives
$$
\widetilde{F}_n^{(j,0)}(t,t)\equiv 0, \quad j=1,\dots,2n-2, 
$$
and
$$
 \widetilde{F}_n^{(2n-1,0)}(t,t) = \widetilde{\beta}_1^{(1,0)}(t,t)\cdots \widetilde{\beta}_n^{(1,0)}(t,t) \not\equiv 0.
$$
Thus, by Lemma~\ref{InverseFormF}, the $\ast$-inverse of $F_n$ exists and takes on the form 
$$
F_{n}^{\ast-1}=\delta^{(2n+1)}\ast F_{inv}.
 $$
With the further assumption that $\widetilde{\beta}_j^{(1,0)}(t,t)\neq 0$ for every $t \in I$, $j=1,\dots,n$, we get $F_{inv}\in Sm_\Theta(I)$.
We can now return to calculating $\beta_{n+1}$. We start with
\begin{align*}
\beta_{n+1}=F_{n+1}\ast F_n^{\ast-1}=F_{n+1}\ast\delta^{(2n+1)}\ast F_{inv}.
\end{align*}
By Eq.~(\ref{Fn1Result}), we have $F_{n+1}\ast \delta^{(2n+1)}=\widetilde{F}^{(0,2n+1)}_{n+1}\Theta$ and hence
\begin{align}
\beta_{n+1}&=\big(\widetilde{F}^{(0,2n+1)}_{n+1}\Theta\big)\ast F_{inv},\nonumber\\
&=\int_t^{t'} \widetilde{F}^{(0,2n+1)}_{n+1}(t',\tau)\widetilde{F}_{inv}(\tau,t)d\tau~\Theta,\label{betaFormSimple}
\end{align}
because $F_{inv}\in \Sm$.
This shows that $\beta_{n+1}\in \Sm$ is piecewise smooth. Furthermore, in the limit $t'\to t$, the integral above vanishes since the integrand is smooth, and $\widetilde{F}^{(0,2n)}_{n+1}(t,t)$ is identically null by Eq.~(\ref{Fn1Deriv}), consequently $\beta_{n+1}(t,t)\equiv 0$.
Since neither $\widetilde{\beta}_{n+1}^{(1,0)}(t,t)$ nor $\widetilde{\beta}_{n+1}^{(0,1)}(t',t')$ are identically null by Lemma~\ref{lemma:breakdown}, the proof is concluded.
As a final remark, note that
$\beta_{n+1}^{\ast-1}$ exists and is of the form $\beta_{n+1}^{\ast-1}=\delta^{(3)}\ast b_{n+1}$, with $b_{n+1}\in\Sm$ given explicitly by Lemma~\ref{InverseFormF}.\\ 

These results establish $\mathfrak{P}_n\Rightarrow \mathfrak{P}_{n+1}$ and, since $\mathfrak{P}_1$ holds, $\mathfrak{P}_n$ is true for $n=1,2,\dots,N-1$. 
\end{proof}

\begin{proof}[Proof for the coefficients $\alpha$]
 A completely similar proof works for the $\alpha_j$ coefficients, on invoking  auxiliary matrices $\mathsf{Q}_n$ defined as equal to $\mathsf{T}_n$ but for $\alpha_{n-1}$ set to $0$. Then
$$
G_n(t',t):=m_{2n+1}(t',t)-\big(\mathsf{Q}_{n+1}^{\ast 2n+1}\big)_{11}=\alpha_{n}\ast\beta_{n}\ast\cdots \ast \beta_1
$$
and furthermore $G_n\in \Sm$. Since now $\alpha_n=G_n\ast F_n^{\ast -1}$,
then
\begin{align*}
    \alpha_n=G_n\ast F_n^{\ast -1} 
        &= \left(G_n \ast \delta^{(2n)}\right)\ast \delta' \ast F_{inv} \\
        &= \left(G_n^{(0,2n)}(t',t)\Theta \right)\ast \delta' \ast F_{inv} \\
        &= \left(G_n^{(0,2n+1)}(t',t)\Theta + G_n^{(0,2n)}(t,t)\delta \right)\ast F_{inv}.
\end{align*}
Hence $\alpha_n\in \Sm$, however $\alpha_n(t,t)$ may be not identically null.
\end{proof}

\section{Conclusion}
In this work, we have shown that any systems of coupled linear differential equations \eqref{FundamentalSystem} with smooth coefficients can be transformed into a smooth tridiagonal system, for $t',t \in I$, when the matrix of coefficients $\widetilde{\mathsf{A}}(t')$ is tridiagonalizable in the usual sense for every $t' \in I$.
In particular, baring accidental breakdowns of the $\ast$-Lanczos algorithm, we showed that for any matrix $\mathsf{A}(t')$ composed of smooth functions of $t'$ and for any two vectors $\b{v}$ and $\b{w}$, there exists a \emph{tridiagonal} matrix comprising only piecewise smooth functions and non-essential Dirac delta distributions (Remark~\ref{Rem:nonessential}) whose ordered exponential evaluated between $\b{w}^H$ and $\b{v}$ yields the same result as the ordered exponential of $\mathsf{A}$ evaluated between these two vectors. Moreover, we proved that a sufficient condition for not having a breakdown in the $\ast$-Lanczos algorithm is that the usual non-Hermitian Lanczos algorithm with inputs $\mathsf{A}(t'), \b{w}, \b{v}$ does not breakdown for at least one $t' \in I$. If, however, there exists $\rho$ so that the inputs $\mathsf{A}(\rho), \b{w}, \b{v}$ do produce a breakdown in the usual non-Hermitian Lanczos algorithm, then the smoothness of the resulting tridiagonal matrix is not guaranteed. In this case, the $\ast$-Lanczos algorithm can proceed by restricting the given domain $I$ to a subset excluding $\rho$.
Given the pervasive presence of systems of coupled linear differential equations with non-constant coefficients in applications--for example all closed quantum dynamical systems subjected to time-dependent forces produce such a system--this result provides an essential basis for the  evaluation and understanding of ordered exponentials. Concretely, the proofs provided  here guarantee the existence and good-behavior of a constructive procedure, the $\ast$-Lanczos algorithm, capable of exactly evaluating ordered exponentials in a finite number of steps.

\section*{Acknowledgments}
This work has been supported by Charles University Research program No. UNCE/SCI/023 and by the 2019 ANR JCJC \textsc{Alcohol} project ANR-19-CE40-0006.

\bibliographystyle{amsplain} 
\bibliography{references}

\providecommand{\bysame}{\leavevmode\hbox to3em{\hrulefill}\thinspace}
\providecommand{\MR}{\relax\ifhmode\unskip\space\fi MR }
\providecommand{\MRhref}[2]{%
  \href{http://www.ams.org/mathscinet-getitem?mr=#1}{#2}
}
\providecommand{\href}[2]{#2}
\begin{thebibliography}{10}

\bibitem{Abou2003}
Hisham {Abou-Kandil}, Gerhard Freiling, Vlad Ionescu, and Gerhard Jank,
  \emph{Matrix {{Riccati Equations}} in {{Control}} and {{Systems Theory}}},
  Systems \& {{Control}}: {{Foundations}} \& {{Applications}}, {Birkh{\"a}user
  Basel}, 2003 (en).

\bibitem{Autler1955}
S.~H. Autler and C.~H. Townes, \emph{Stark effect in rapidly varying fields},
  Phys. Rev. \textbf{100} (1955), 703--722.

\bibitem{BenEtAll17}
P.~Benner, A.~Cohen, M.~Ohlberger, and K.~Willcox, \emph{Model reduction and
  approximation: Theory and algorithms}, Computational Science and Engineering,
  SIAM, 2017.

\bibitem{Blanes2009}
S.~Blanes, F.~Casas, J.A. Oteo, and J.~Ros, \emph{The magnus expansion and some
  of its applications}, Physics Reports \textbf{470} (2009), no.~5, 151 -- 238.

\bibitem{Blanes15}
Sergio Blanes, \emph{{High order structure preserving explicit methods for
  solving linear-quadratic optimal control problems}}, Numer. Algorithms
  \textbf{69} (2015), no.~2, 271--290.

\bibitem{Casas07}
Fernando Casas, \emph{Sufficient conditions for the convergence of the magnus
  expansion}, Journal of Physics A: Mathematical and Theoretical \textbf{40}
  (2007), no.~50, 15001--15017.

\bibitem{Corless2003}
Martin {Corless} and Art {Frazho}, \emph{{Linear Systems and Control: An
  Operator Perspective}}, {Pure and Applied Mathematics}, {CRC Press}, 2003
  (en).

\bibitem{Feldman1984}
E.~B. Fel'dman, \emph{On the convergence of the magnus expansion for spin
  systems in periodic magnetic fields}, Physics Letters \textbf{104A} (1984),
  no.~9, 479--481.

\bibitem{Giscard2015}
P.-L. Giscard, K.~Lui, S.~J. Thwaite, and D.~Jaksch, \emph{An exact formulation
  of the time-ordered exponential using path-sums}, Journal of Mathematical
  Physics \textbf{56} (2015), no.~5, 053503.

\bibitem{GisPozInv19}
P.-L. Giscard and S.~Pozza, \emph{Lanczos-like algorithm for the time-ordered
  exponential: The $\ast$-inverse problem}, Accepted in Applications of
  Mathematics, to appear. Preprint arXiv:1910.05143 [math.NA] (2020).

\bibitem{GiscardPozza2019}
P.-L. {Giscard} and S.~{Pozza}, \emph{{Lanczos-like method for the time-ordered
  exponential}},  \textbf{arXiv:1909.03437 [math.NA]} (2020).

\bibitem{GiscardBonhomme2019}
Pierre-Louis {Giscard} and Christian {Bonhomme}, \emph{{General solutions for
  quantum dynamical systems driven by time-varying Hamiltonians: applications
  to NMR}}, arXiv e-prints (2019), arXiv:1905.04024.

\bibitem{ProceedingsPaper2019}
Pierre-Louis Giscard and Stefano Pozza, \emph{Lanczos-like algorithm for the
  time-ordered exponential: The $\ast$-inverse problem}, 2019.

\bibitem{hached2018}
M.~Hached and K.~Jbilou, \emph{Numerical solutions to large-scale differential
  {{Lyapunov}} matrix equations}, Numerical Algorithms \textbf{79} (2018),
  no.~3, 741--757 (en).

\bibitem{schwartz1952}
Israel Halperin and Laurent Schwartz, \emph{Introduction to the theory of
  distributions}, University of Toronto Press, Toronto, 19 Feb. 2019.

\bibitem{IseAl2000}
Arieh Iserles, Hans~Z. Munthe-Kaas, Syvert~P. Nørsett, and Antonella Zanna,
  \emph{Lie-group methods}, Acta Numerica \textbf{9} (2000), 215–365.

\bibitem{Flum2004}
M.~Grohe J.~Flum, \emph{{The Parameterized Complexity of Counting Problems}},
  SIAM Journal on Computing \textbf{33} (2004), 892--922.

\bibitem{KirSim19}
Gerhard {Kirsten} and Valeria {Simoncini}, \emph{{Order reduction methods for
  solving large-scale differential matrix Riccati equations}},
  \textbf{arXiv:1905.12119 [math.NA]} (2019).

\bibitem{kosovtsov2002introduction}
Yu.~N. Kosovtsov, \emph{The introduction to the operator method for solving
  differential equations.first-order de}, 2002.

\bibitem{kosovtsov2004chronological}
\bysame, \emph{The chronological operator algebra and formal solutions of
  differential equations}, 2004.

\bibitem{kosovtsov2009formal}
\bysame, \emph{Formal exact operator solutions to nonlinear differential
  equations}, 2009.

\bibitem{kuvcera73}
Vladim{\'\i}r Ku{\v{c}}era, \emph{A review of the matrix riccati equation},
  Kybernetika \textbf{9} (1973), no.~1, 42--61.

\bibitem{kwaSiv72}
Huibert Kwakernaak and Raphael Sivan, \emph{Linear optimal control systems},
  vol.~1, Wiley-interscience New York, 1972.

\bibitem{Lauder1986}
M.A. Lauder, P.L. Knight, and P.T. Greenland, \emph{Pulse-shape effects in
  intense-field laser excitation of atoms}, Optica Acta: International Journal
  of Optics \textbf{33} (1986), no.~10, 1231--1252.

\bibitem{Magnus1954}
Wilhelm Magnus, \emph{On the exponential solution of differential equations for
  a linear operator}, Communications on Pure and Applied Mathematics \textbf{7}
  (1954), no.~4, 649--673.

\bibitem{Maricq1987}
M.~Matti Maricq, \emph{Convergence of the magnus expansion for time dependent
  two level systems}, The Journal of Chemical Physics \textbf{86} (1987),
  no.~10, 5647--5651.

\bibitem{Par92}
Beresford~N. Parlett, \emph{{Reduction to tridiagonal form and minimal
  realizations}}, SIAM J. Matrix Anal. Appl. \textbf{13} (1992), no.~2,
  567–593. \MR{MR1152769 (93c:65059)}

\bibitem{PozPra19}
Stefano Pozza and Miroslav~S. Prani\'c, \emph{{The {G}auss quadrature for
  general linear functionals, {L}anczos algorithm, and minimal partial
  realization}}, arXiv e-prints (2019).

\bibitem{PozPraStr18}
Stefano Pozza, Miroslav~S. Prani\'c, and Zden\v{e}k Strako\v{s}, \emph{{The
  {L}anczos algorithm and complex {G}auss quadrature}}, Electron. Trans. Numer.
  Anal. \textbf{50} (2018), 1--19.

\bibitem{Reid63}
William~T Reid, \emph{Riccati matrix differential equations and non-oscillation
  criteria for associated linear differential systems}, Pacific J. Math.
  \textbf{13} (1963), no.~2, 665--685.

\bibitem{Sanchez2011}
S.~S{\'a}nchez, F.~Casas, and A.~Fern{\'a}ndez, \emph{New analytic
  approximations based on the magnus expansion}, Journal of Mathematical
  Chemistry \textbf{49} (2011), no.~8, 1741--1758.

\bibitem{schwartz1978}
Laurent Schwartz, \emph{Th{\'e}orie des distributions}, nouvelle {\'e}dition,
  enti{\`e}rement corrig{\'e}e, refondue et augment{\'e}e ed., {Hermann},
  {Paris}, 1978.

\bibitem{Shirley1965}
Jon~H. Shirley, \emph{Solution of the schr\"odinger equation with a hamiltonian
  periodic in time}, Phys. Rev. \textbf{138} (1965), B979--B987.

\bibitem{Volterra1928}
V.~{Volterra} and J.~{Pérès}, \emph{{Leçons sur la composition et les
  fonctions permutables}}, {Éditions Jacques Gabay}, 1928 (en).

\bibitem{Xie2010}
Qiongtao Xie and Wenhua Hai, \emph{Analytical results for a monochromatically
  driven two-level system}, Phys. Rev. A \textbf{82} (2010), 032117.

\end{thebibliography}

\end{document}